\documentclass[12pt]{amsart}
\usepackage{geometry} 
\usepackage{graphicx}
\usepackage{amsfonts}
\usepackage{amsthm}
\usepackage{amsmath}
\usepackage{amssymb}
\usepackage[arrow,matrix,curve]{xy}
\usepackage{enumerate}
\usepackage{color}
\usepackage[normalem]{ulem}
\usepackage{tikz}
\usepackage{xspace}
\definecolor{light-blue}{rgb}{0.8,0.85,1}
\definecolor{light-red}{rgb}{1,.4,.4}
\definecolor{purp}{rgb}{.7,.3,1}
\definecolor{yel}{rgb}{1,1,.5}
\definecolor{cy}{rgb}{0,1,1}
\usepackage{colortbl}
\usepackage{multirow}
\usepackage{array}
\newtheorem{theorem}{Theorem}

\newtheorem{proposition}[theorem]{Proposition}
\theoremstyle{definition}

\newtheorem{definition}[theorem]{Definition}

\newtheorem{remark}[theorem]{Remark}
\newcommand{\co}{\colon\,}
\newcommand{\bA}{\mathbb A}

\newcommand{\bR}{\mathbb R}
\newcommand{\bC}{\mathbb C}
\newcommand{\bE}{\mathbb E}
\newcommand{\bH}{\mathbb H}

\newcommand{\bP}{\mathbb P}

\newcommand{\cM}{\mathcal M}

\newcommand{\cP}{\mathcal P}
\newcommand{\cS}{\mathcal S}

\renewcommand\Im{\operatorname{Im}}

\title[Conics in hyperbolic geometry]{Different definitions of conic
  sections\\ in hyperbolic geometry}
\author{Patrick Chao}
\address{Montgomery Blair High School\\
51 University Blvd E\\
Silver Spring, MD 20901-2451, USA}
\curraddr{University of California, Berkeley\\
Berkeley, CA 94720}
\email[Patrick Chao]{prc@berkeley.edu}
\author{Jonathan Rosenberg}
\address{Department of Mathematics\\
University of Maryland\\
College Park, MD 20742-4015, USA} 
\email[Jonathan Rosenberg]{jmr@math.umd.edu}
\thanks{JR partially supported by NSF grants DMS-1206159, DMS-1607162.  This
  paper is based on a summer research project by PC under the
  supervision of JR as part of the Montgomery Blair High School summer
research internship program in 2015.}  
  
\begin{document}
\begin{abstract}
In classical Euclidean geometry, there are several equivalent
definitions of conic sections. We show that in the hyperbolic plane,
the analogues of these same definitions still make sense, but are no
longer equivalent, and we discuss the relationships among them.  
\end{abstract}

\maketitle

\section{Introduction}
\label{sec:intro}
Throughout this paper, $\bE^n$ will denote Euclidean $n$-space and
$\bH^n$ will denote hyperbolic $n$-space. Recall that (up to isometry)
these are the unique complete simply connected
Riemannian $n$-manifolds with constant
curvature $0$ and $-1$, respectively. We will use $d(x,y)$ for the
Riemannian distance between points $x$ and $y$ in either of these
geometries. We will sometimes identify
$\bE^n$ with affine $n$-space $\bA^n(\bR)$ over the reals, which can
then be embedded as usual in projective $n$-space $\bP^n(\bR)$ (the
set of lines through the origin in $\bA^{n+1}(\bR)$).
While this paper is about $2$-dimensional geometry, we will sometimes
need to consider the case $n=3$ as well as $n=2$.

\subsection{Hyperbolic geometry}
\label{sec:hyperbolic}
\emph{Hyperbolic geometry} is a form of non-Euclidean geometry, which
modifies Euclid's fifth axiom, the parallel postulate. The parallel
postulate has an equivalent statement, known as \emph{Playfair's axiom}. 

\begin{definition}[Playfair's axiom]
\label{def:Playfair}
Given a line $l$ and a point $p$ not on $l$, there exists only one
line through $p$ parallel to $l$. 
\end{definition}

In hyperbolic geometry, this is modified by allowing an infinite
number of lines through $p$ parallel to $l$. This has interesting
effects, resulting in the angles in a triangle adding up to less than
$\pi$ radians, and a relation between the area of the triangle
and the \emph{angular defect}, the difference between $\pi$
and the sum of the angles. Hyperbolic
geometry may also be considered to be the Riemannian geometry
of a surface
of constant negative curvature. When this curvature is normalized to
$-1$, there are especially nice formulas, such as the fact that the
area of a triangle is equal to the angular defect. 

However, since a surface of negative curvature cannot be embedded in a
surface of zero curvature, hyperbolic geometry requires ``models'' to
represent hyperbolic space on a flat sheet of paper. There are many
such models, including the Poincar\'e disk and the Poincar\'e upper half-plane
model. These all have varying metrics and methods of representing
lines (i.e., geodesics)  and shapes. 

In the \emph{Poincar\'e disk model} of $\bH^2$, $\{z\in \bR^2\mid \vert z\vert
  <1\}$, geodesics are either circular arcs orthogonal to the unit
circle or else lines through the origin. The metric is defined as: 
\[
(ds)^2=\frac{(dx)^2+(dy)^2}{(1-x^2-y^2)^2}.
\]

In the \emph{upper half-plane model} of $\bH^2$,
$\{(x,y)\in \bR^2\mid   y>0\}$,
geodesics are either lines orthogonal to the $x$-axis or
else circular arcs orthogonal to the $x$-axis. The metric is defined as: 
\[
(ds)^2=\frac{(dx)^2+(dy)^2}{y^2}.
\]

In the \emph{Klein disk model} or \emph{Beltrami-Klein model} of $\bH^2$,
the points in the model are the points of the open unit disk in the
Euclidean plane, and the geodesics are the intersection with the open disk
of chords joining two points on the unit circle.  The formula for
the metric in this model is rather complicated:
\[
(ds)^2=\frac{(dx)^2+(dy)^2}{1-x^2-y^2}+\frac{(x\,dx+y\,dy)^2}{(1-x^2-y^2)^2}.
\]

\subsection{Conics}
\label{sec:conics}
\begin{definition}
  \label{def:Euclconics}
One of the oldest notions in
geometry, going all the way back to Apollonius, is that of \emph{conic
  sections} in $\bE^2$. There are at least four equivalent
definitions of a conic section $C$:
\begin{enumerate}
\item a smooth irreducible algebraic curve in $\bA^2(\bR)$ of degree $2$;
\item the intersection of a right circular
  cone in $\bE^3$ (with vertex at the origin, say)
  with a plane not passing through the origin, this plane in turn
  identified with $\bE^2$;
\item the \emph{two focus definition}: Fix two points $a_1, a_2\in
  \bE^2$. An \emph{ellipse} $C$ is the locus of points $x\in \bE^2$ such that
  $d(x,a_1) + d(x,a_2) = c$, where $c>0$ is a fixed
  constant. Similarly, a \emph{hyperbola} $C$ is the locus of points
  $x\in \bE^2$ such that
  $\left\vert d(x,a_1) - d(x,a_2)\right\vert =   c$,
  where $c>0$ is a fixed constant. The points $a_1$ and $a_2$ are
  called the \emph{foci} of the conic, and the line joining
  them (assuming $a_1\ne a_2$) is called the \emph{major axis}.   A
  \emph{circle} is the 
  special case of an ellipse where $a_1=a_2$. A \emph{parabola} is the
  limiting case of a one-parameter family of ellipses $C_t(a_1,a_2,c_t)$ where
  $a_1$ is fixed and we let $a_2$ run off to infinity along the major
  axis keeping $c_t-d(a_1,a_2)$ fixed.
\item the \emph{focus/directrix definition}: Fix a point $a_1\in
  \bE^2$, called the \emph{focus}, and a line $\ell$ not passing
  through $a_1$, called the \emph{directrix}. A \emph{conic} $C$ is
  the locus of points with $d(x,a_1)=\varepsilon d(x, \ell)$, where
  $\varepsilon >0$ is a constant called the \emph{eccentricity}. If
  $\varepsilon <1$ the conic is 
  called an \emph{ellipse}; if $\varepsilon=1$ the conic is
  called a \emph{parabola}; if $\varepsilon>1$ the conic is
  called a \emph{hyperbola}.  A \emph{circle} is the limiting case of an
  ellipse obtained by fixing $a_1$ and sending $\varepsilon\to 0$ and $d(a_1,
  \ell)\to \infty$ while keeping $r= \varepsilon d(a_1, \ell)$ fixed.
\end{enumerate}
Note that these definitions come from totally different realms.
Definition
\ref{def:Euclconics}(1) is from algebraic geometry. Definition
\ref{def:Euclconics}(2) uses a
totally geodesic embedding of $\bE^2$ into $\bE^3$.  Definitions
\ref{def:Euclconics}(3)
and \ref{def:Euclconics}(4) use only the metric geometry of $\bE^2$.

Since Definition \ref{def:Euclconics}(1) is phrased in terms of
algebraic geometry, it 
naturally leads to a definition of a \emph{conic in $\bP^2(\bR)$} as a smooth
irreducible algebraic curve of degree $2$. Such a curve must be given
(in homogeneous coordinates) by a homogeneous quadratic equation
$Q(x)=0$, where $Q$ is a nondegenerate indefinite quadratic form on
$\bR^3$. This is the equation of a cone, and intersecting the cone
with an affine plane not passing through the origin (the vertex of the
cone) gives us back Definition \ref{def:Euclconics}(2). 
\end{definition}

\subsection{Contents of this paper}
\label{sec:contents}
The topic of this paper is studying what happens to Definitions
\ref{def:Euclconics}(1)--(4) when we replace $\bE^2$ by $\bH^2$.
This is an old problem, and is discussed for example in
\cite{MR1505334,MR1628013,MR0095442,MR0171205,MR487028}.  However, as
we will demonstrate, the analogues of Definitions
\ref{def:Euclconics}(1)--(4) are 
no longer equivalent in $\bH^2$.  Thus there is some confusion in the
literature, and those who talk about conic sections in $\bH^2$ (as
recently as \cite{MR3274526,2015arXiv150406450}) do not
always all mean the same thing.  Our main results are
Theorems \ref{thm:twocircles}, \ref{thm:twoellipses},
\ref{thm:parabolas}, and \ref{thm:twohyps} in Section
\ref{sec:results}, which clarify the relationships among these
definitions (especially the two-focus and focus-directrix definitions)
in $\bH^2$.  Table \ref{table} summarizes our results.
\begin{table}[hbt]
\begin{center}
\begin{tabular}{|l|}
\hline
\multicolumn{1}{|c|}{\textbf{Circles}}\\
\ref{def:circ} $\Leftrightarrow$ \ref{def:twofocushypconic}. \ref{def:circ},
\ref{def:twofocushypconic} included in \ref{def:hypconicinH3},
\ref{def:Molnar}. \ref{def:circ},
\ref{def:twofocushypconic} $\not\Leftrightarrow$ 
\ref{def:focdirhypconic}.\\
\hline
\multicolumn{1}{|c|}{\textbf{Horocycles (paracycles), hypercycles}}\\
included in \ref{def:alghypconic}, \ref{def:Molnar}.
not included in \ref{def:twofocushypconic},
\ref{def:focdirhypconic}.\\ 
\hline
\multicolumn{1}{|c|}{\textbf{Ellipses}}\\
\ref{def:twofocushypconic} $\not\Leftrightarrow$ 
\ref{def:focdirhypconic}. But when \ref{def:focdirhypconic} gives a
closed curve, it is included in \ref{def:twofocushypconic}.\\
\hline
\multicolumn{1}{|c|}{\textbf{Hyperbolas}}\\
\ref{def:twofocushypconic} $\subsetneqq$ 
\ref{def:focdirhypconic}.\\
\hline
\multicolumn{1}{|c|}{\textbf{Parabolas}}\\
\ref{def:twofocushypconic} $\not\Leftrightarrow$ 
\ref{def:focdirhypconic}. Neither kind of parabola is ever closed.\\
\hline
\end{tabular}
\medskip
\caption{Relationships among possible definitions of conics in
  $\bH^2$. Numbers refer to the numbered definitions in sections
  \ref{sec:axioms} and \ref{sec:results}.}
\label{table}
\end{center}
\end{table}

\section{Other Axiomatizations}
\label{sec:axioms}

Before discussing conics in $\bH^2$, we first explain still another
definition of conic sections in $\bP^2(\bR)$, which is the definition
found in \cite[Ch.\ III]{MR1628013} and with a slight variation in
\cite{MR1505334}. This definition uses the notion of a \emph{polarity}
$p$ in $\bP^2(\bR)$. This is a particular type of
mapping of points to lines and lines to
points preserving the incidence relations of
projective geometry (or in the language of \cite[\S3.1]{MR1628013}, a
\emph{correlation}).  It can be explained in 
terms of algebraic geometry as follows. If $Q$ is a nondegenerate
quadratic form on $\bR^3$, then there is an associated nondegenerate
symmetric bilinear form defined by $B(x,y)= (Q(x+y)-Q(x)-Q(y))/2$, and
if $V$ if a linear subspace of $\bR^3$ of dimension $d=1$ or $2$, then
the orthogonal complement $V^{\perp, B}$ of $V$ with respect to $B$ is a linear
subspace of dimension $3-d$.  Thus the process $p$ of taking orthogonal
complements with respect to $B$ sends points in $\bP^2(\bR)$, which
are $1$-dimensional linear subspaces of $\bR^3$, to lines (copies of
$\bP^1(\bR)$), which are $2$-dimensional linear subspaces of $\bR^3$,
and \emph{vice versa}. Given a polarity $p$, the associated
\emph{conic} is the set $C$ of points $x\in \bP^2(\bR)$ such that $x$
lies on the line $p(x)$, i.e., the set of $1$-dimensional linear
subspaces $V$ of $\bR^3$ for which $V\subset V^{\perp, B}$, or in
other words, for which $V$ is $B$-isotropic.  Thus if
we identify the point $x\in \bP^2(\bR)$ with its homogeneous
coordinates, or with a basis vector for $V$ up to rescaling, this
becomes the condition $B(x,x)=0$, or $Q(x)=0$, which is just
Definition \ref{def:Euclconics}(1).  (Note that if $Q$ is definite,
the conic is empty, so we are forced to take $Q$ to be indefinite in
order to get anything interesting.)  Conversely, it is well known 
\cite[\S4.72]{MR1628013} that
every polarity arises from a nonsingular symmetric matrix or
equivalently from a nondegenerate quadratic form $Q$, so the polarity
definition of conics in \cite[Ch.\ III]{MR1628013} is equivalent to 
Definition \ref{def:Euclconics}(1). 

We now introduce several possible definitions of conic sections in
$\bH^2$.
\begin{definition}[A metric circle]
\label{def:circ}
A \emph{circle} in $\bH^2$ is the locus of points a fixed distance
$r>0$ from a \emph{center} $x_1\in \bH^2$, i.e., 
$C=\{x\in \bH^2: d(x,x_1)=r\}$. 
\end{definition}
\begin{definition}[{Analogue of Definition \ref{def:Euclconics}(2)}]
\label{def:hypconicinH3}
A \emph{right circular cone} in $\bH^3$ is defined as follows.  Fix a point
$x_0\in \bH^3$ (say the origin, if we are using the standard unit ball in
$\bR^3$ as our model of $\bH^3$) and fix a plane $P$ in $\bH^3$
(a totally geodesic copy of $\bH^2$) not passing through $x_0$.  
There is a unique ray starting at $x_0$ and intersecting $P$
perpendicularly. Let $x_1$ be the intersection point (the closest
point on $P$ to $x_0$), and fix a radius $r>0$. We then have the circle
$C$ in $P$ centered at $x_1$ with radius $r$.  The cone $c(x_0, C)$
through $x_0$ and $C$ is then the union of
the lines (geodesics) through $x_0$ passing through a point of $C$.
The point $x_0$ is called the \emph{vertex} of the cone.
A \emph{conic section} (in the literal sense!)
in $\bH^2$ is then the intersection of a plane $P'$ in $\bH^3$ (not
passing through $x_0$) with $c(x_0, C)$.
\end{definition}
Since we can take $P'=P$ in the above definition, it is obvious that a
circle (as in Definition \ref{def:circ}) is a special case of a conic
section in the sense of Definition \ref{def:hypconicinH3}.

In the Poincar\'e ball 
model of $\bH^3$ with $x_0$ the origin, geodesics through
$x_0$ are just straight lines for the Euclidean metric, so it's easy
to see that a right circular cone with vertex $x_0$ is also a right
circular cone in the Euclidean sense in $\bR^3$.  On the other hand,
planes in $\bH^3$ not passing through $x_0$ correspond to Euclidean
spheres perpendicular to the unit sphere (the boundary of the model of
$\bH^3$).  Thus a conic section in the sense of Definition
\ref{def:hypconicinH3} is the intersection of a right circular cone
with a sphere, and is thus (in terms of the algebraic geometry of
$\bA^3(\bR)$) an algebraic curve of degree $\le 4$. To view this conic
in the usual Poincar\'e disk model of $\bH^2$, we apply an 
isometry (stereographic projection) from $P$ to the unit disk in
$\bC$. Since this is a rational map, we see that any conic section in
the sense of Definition \ref{def:hypconicinH3} is an algebraic curve
(in fact of degree $\le 4$) when viewed in the disk model of $\bH^2$.
Alternatively, if we use the Klein ball model of $\bH^3$ with $x_0$ the
origin, then a right circular cone with vertex $x_0$ will again look
like a Euclidean right circular cone, while a $2$-plane in $\bH^3$
will be the intersection of the ball with 
a Euclidean $2$-plane, and any conic section in the
sense of Definition \ref{def:hypconicinH3} will also be a conic
section in the Euclidean sense of Definition \ref{def:Euclconics}.
Thus Definition \ref{def:hypconicinH3} is equivalent to the following
Definition \ref{def:alghypconic}.
\begin{definition}[{Analogue of Definition \ref{def:Euclconics}(1)}]
A \emph{conic in $\bH^2$ in the algebraic sense} is the intersection of a
smooth irreducible algebraic curve of degree $2$
in $\bA^2(\bR)$ with the open unit
disk, viewed as the Klein disk model for $\bH^2$. (This is a
non-conformal model in which points of $\bH^2$ are points of the open
unit disk, and the straight lines are intersections with
the open disk of straight lines in the plane.) Such a conic is
closed (compact) if and only if it is a circle or ellipse not
intersecting the unit circle (the \emph{absolute} in the terminology
of \cite{MR1505334} and \cite{MR1628013}).
\label{def:alghypconic}
\end{definition}

Definition \ref{def:alghypconic} is the definition of conics used in
\cite{MR1505334} and \cite{MR1628013}.

Still another approach to defining conics may be found in \cite{MR487028},
based on the axiom system for $\bE^2$ and $\bH^2$ developed in
\cite{MR0346643}. First we need to discuss Bachmann's approach to
metric geometry. Bachmann observes that in either $\bE^2$ and $\bH^2$, there
is a unique isometry which is reflection in a given line $a$ or around
a given point $A$.  Thus we can identify lines and points with certain
distinguished involutory elements $\cS$ (the reflections in lines)
and $\cP$  (the reflections around points) of the isometry group
$G$.  More is true: every element of $G$ is a product of at most three
elements of $\cS$.  Elements of $\cS$ are orientation-reversing; elements
of $\cP$ are orientation-preserving. The product of two elements
$a,\,b\in \cS$ is a non-trivial involution if and only if $a\ne b$ and
$ab=ba$; in this case, the lines associated to $a$ and $b$ are
perpendicular (we write $a\perp b$)
and $ab\in \cP$ is the reflection around the unique intersection point
of $a$ and $b$. Furthermore, every element of $G$ of order $2$ belongs
to $\cS$ or to $\cP$, but not to both.  A point $A\in \cP$ lies on a
line $a\in \cS$ exactly when there exists $b\in \cS$ commuting with
$a$ such that $A = ab$.  Thus a \emph{metric plane} $\cM$
can be identified with a
group $G$ together with a distinguished generating set $\cS$
consisting of involutions and the set $\cP$ of non-trivial products of
commuting elements of $\cS$, satisfying certain axioms.
We won't need the axioms here, since they will be evident in the
cases we are interested in. In the case of $\bE^2$, $G=\bR^2\rtimes
O(2)$, the usual Euclidean 
motion group, and in the case of $\bH^2$, $G=PGL(2,\bR)\cong
O^+(2,1)$. (If $\begin{pmatrix} a&b\\c&d\end{pmatrix}$ has determinant
$+1$, then it operates on the upper half-plane by linear fractional
transformations, which are orientation-preserving, and if it has
determinant $-1$, then it operates on the upper half-plane by
\[
z\mapsto \frac{a\bar z + b}{c\bar z + d},
\]
and this conjugate-linear map is an orientation-reversing isometry of
$\bH^2$.) 

Bachmann also points out that the metric plane $(\cM,
\cS, \cP)$ corresponding to $\bE^2$ or $\bH^2$ can be embedded naturally in a
\emph{projective-metric plane} $(\cP\cM, \cS', \cP')$, in such a way
that $\cS\subseteq \cS'$ and $\cP\subseteq \cP'$.
In the case of $\bE^2$, this is just the usual embedding of
$\bA^2(\bR)$ in $\bP^2(\bR)$ by adjoining a copy of $\bP^1(\bR)$ at
$\infty$, and the associated group is $PGL(3,\bR)$.  In the case of
$\bH^2$, $\cP\cM$ is again a copy of $\bP^2(\bR)$, but its points 
and lines consist of \emph{ideal points} and \emph{ideal lines} of
$\bH^2$. A simple way to visualize the embedding of $\bH^2$ in
$\bP^2(\bR)$ is to use the (non-conformal) Klein model of $\bH^2$, in
which points are points in the interior of the unit disk in $\bR^2$,
and lines are the intersections of ordinary straight lines in
$\bA^2(\bR)$ with the unit disk.  Then each point or line of $\bH^2$
obviously corresponds to a unique point or line of $\bP^2(\bR)$.
When viewed as $\cP\cM$ in this way, $\bP^2(\bR)$ carries a canonical
polarity, namely the one associated to the unit circle in
$\bA^2(\bR)$, viewed as a conic in the sense of the polarity
definition at the beginning of this Section.  When we embed
$\bA^2(\bR)$ in $\bP^2(\bR)$ as usual via $(x,y)\mapsto [x,y,1]$
(homogeneous coordinates denoted by square brackets), this polarity is
associated to the quadratic form $Q\co (x,y,z)\mapsto x^2+y^2-z^2$, since
$Q(\cos\theta,\sin\theta,1)=0$ for any real angle $\theta$.
\begin{definition}[{\cite[Definition 4.1]{MR487028}}]
\label{def:Molnar}
A \emph{conic $C$ in the sense of Moln\'ar}, with foci $A,\,B\in
\bP^2(\bR)$, is defined by choosing a line $x_1$ in $\bP^2(\bR)$ which is not a
boundary line (i.e., $x_1$ is not tangent to the unit circle) and not
passing through either $A$ or $B$, and with $A$ and $B$ not each
other's reflections across $x_1$.  Then $C$ consists of points $X_{11}$
and $X$ chosen as follows.  $X_{11}$ is the intersection of the lines
$a_{11}$ through $A$ and $B^{x_1}$ (the reflection of $B$ across
$x_1$) and $b_{11}$ through $B$ and $A^{x_1}$. (The line $x_1$ is
chosen so that $a_{11}$ and $b_{11}$ are not boundary lines.) The other
points $X$ are defined by fixing a point $Y$ on $x_1$ and
taking the lines $a$ through $Y$ and $A$ and $b$ though $Y$ and $B$,
and then if neither $a$ nor $b$ is a boundary line, letting $X$ be the
intersection of $a_{11}^a$ and $b_{11}^b$ (the reflections of $a_{11}$
and $b_{11}$ across $a$ and $b$, respectively).  Appropriate
modifications are made if $a$ or $b$ is a boundary line.
\end{definition}
As is quite evident, Moln\'ar's definition is quite complicated but
results in a conic section in $\bH^2$ being the intersection of a conic in
$\bP^2(\bR)$ with the unit disk (in the Klein model). We will not
consider this definition further, but it's closely related to
Definition \ref{def:alghypconic}. A picture of the construction with
$A=(0,0)$, $B=(.5,0)$, $x_1=\{y=.5\}$ is shown in Figure
\ref{fig:MolnarConic}. 
\begin{figure}[hbt]
  \begin{center}
    \includegraphics[width=2in]{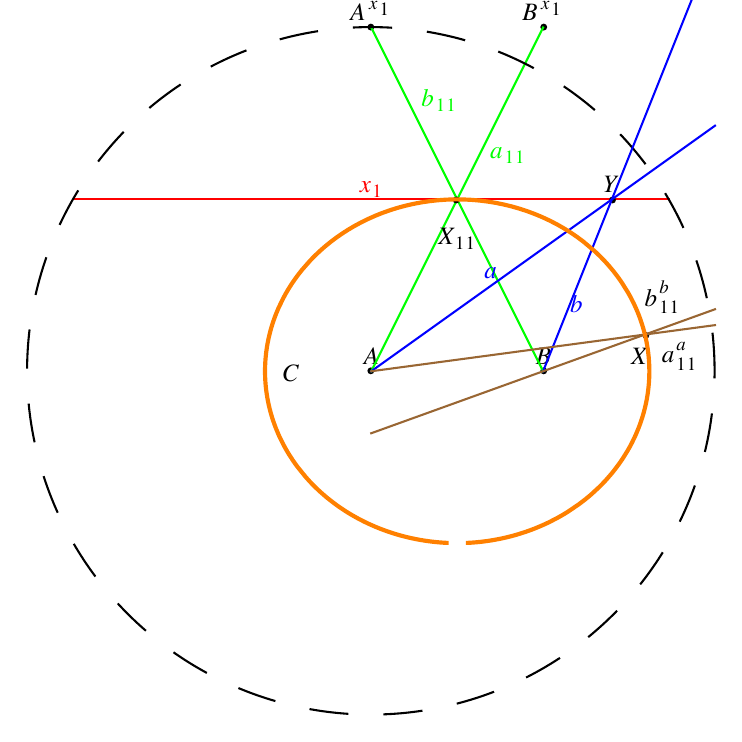}
    \end{center}
\caption{Moln\'ar's construction of a conic}
\label{fig:MolnarConic}
\end{figure}

\section{Main Results}
\label{sec:results}

\begin{definition}[{Analogue of Definition \ref{def:Euclconics}(3)}]
\label{def:twofocushypconic}
The definition of \emph{two focus conics} in Definition
\ref{def:Euclconics}(3) immediately goes over to $\bH^2$, simply by
replacing the Euclidean distance by the hyperbolic distance.  Note
that the case of a circle was already mentioned in
Definition \ref{def:circ}.
\end{definition}
The last definition is the only one that is not immediately
obvious. However, if were we to carry Definition \ref{def:Euclconics}(4)
over to $\bH^2$ without change, then since in the upper half-plane
or disk models of $\bH^2$, the distance function is the log of an
algebraic expression, in the case of irrational eccentricity $\varepsilon$ we
would effectively get the equation 
\[
(\text{algebraic expression}) = (\text{algebraic expression})^\varepsilon,
\]
which is a transcendental equation, and could not possibly agree with
the other definitions of conic sections.  This explains the
modification made in \cite{MR1505334}.  The use of the hyperbolic sine
comes from its role in hyperbolic geometry via the solution of the
Jacobi equation.
\begin{definition}
\label{def:focdirhypconic}[{Analogue of Definition
    \ref{def:Euclconics}(4)}]
Fix a point $a_1\in \bH^2$, called the \emph{focus}, and a line 
(geodesic) $\ell$ not passing   through $a_1$, called the
\emph{directrix}. A \emph{conic} $C$ is the locus of 
  points $x\in \bH^2$ with $\sinh d(x,a_1)=\varepsilon \sinh d(x,
  \ell)$, where $\varepsilon>0$  
  is a constant called the \emph{eccentricity}. If $\varepsilon<1$ the conic is
  called an \emph{ellipse}; if $\varepsilon=1$ the conic is
  called a \emph{parabola}; if $\varepsilon>1$ the conic is
  called a \emph{hyperbola}. (Note: in the case of the parabola, but
  only in this case, the hyperbolic sines cancel and can be removed
  from the definition.) A \emph{circle} is the limiting case of an
  ellipse obtained by fixing $a_1$ and sending $\varepsilon\to 0$ and $d(a_1,
  \ell)\to \infty$ while keeping $r=\varepsilon \sinh d(a_1, \ell)$ fixed.
\end{definition}

\subsection{Circles}
\label{sec:circle}

We begin now to compare the various definitions.  We start with the
circle, which is the most straightforward.  Definition \ref{def:circ}
clearly coincides with Definition \ref{def:hypconicinH3}, in the sense
that if we intersect a right circular cone with a plane perpendicular
to the axis, the result is a circle in the sense of
Definition \ref{def:circ}. We also have the following.
\begin{proposition}
\label{prop:twocircles}
Definition \ref{def:circ} coincides with the case of circles
in Definition \ref{def:alghypconic}, but with the Klein model replaced
by the Poincar\'e model.  In other words, an ordinary
circle in $\bA^2(\bR)$, contained in the open unit disk, 
when viewed as a curve in the Poincar\'e disk model of $\bH^2$
is a metric circle in $\bH^2$, and \emph{vice versa}. Similarly, an
ordinary circle contained in the upper half-plane, when viewed as a
curve in the Poincar\'e upper half-plane model of $\bH^2$, 
is a metric circle in $\bH^2$, and \emph{vice versa}.
\end{proposition}
\begin{proof}
First consider the disk model.  
If the center is the origin, this is clear since the hyperbolic
distance from $0$ to $z$ in $\{z: |z|<1\}$ 
in $\bC$ is a (nonlinear) function $\tanh^{-1}(|z|)=
\frac{1}{2}\log\bigl(\frac{1+|z|}{1-|z|}\bigr)$ of the Euclidean distance
$|z|$ from $0$ to $z$, so that each Euclidean circle centered at $0$ is
also a hyperbolic circle (of a different radius), and \emph{vice
versa}.  However, any circle in $\bH^2$ can be mapped to a
circle centered at $0$ via an isometry of $\bH^2$, and since linear
fractional transformations send circles to circles
\cite[Ch.\ 3, \S3.2, Theorem 14]{MR510197}, the general case
follows. The case of the half-plane model also follows since there is a linear
fractional transformation relating this model to the disk model.
\end{proof}
\begin{remark}
\label{rem:circcenter}
However, one should note that the center of a circle in the unit disk
or the upper half-plane may differ, depending on whether one considers
it as a Euclidean circle or a metric circle in $\bH^2$.  For example,
the metric circle in $\bH^2$ (in the upper-halfplane model) around the
point $i$ with hyperbolic metric radius $\log 2$ has Euclidean equation 
\[
\frac{\vert z - i\vert}{\vert z + i\vert} = \tanh\left(\frac12 \,\log 2\right)
= \frac{1}{3} \text{ or }
\left\vert z - \frac{5}{4}i\right\vert = \frac{3}{4},
\]
so its center as a Euclidean circle is $\frac{5}{4}\, i$.

Metric circles in $\bH^2$, when drawn in the Klein disk model, 
only appear to be circles when centered at the origin. Otherwise, they
are ellipses.
\end{remark}
However, the focus/directrix definition of circles is quite different.
\begin{theorem}
\label{thm:twocircles}
The definition of circle in Definition \ref{def:focdirhypconic} does
\emph{not} agree with the definition of circle in Definitions
\ref{def:circ}, \ref{def:hypconicinH3},
\ref{def:alghypconic}, and \ref{def:twofocushypconic}.
\end{theorem}
\begin{proof}
Consider a circle in the sense of Definition
\ref{def:focdirhypconic}.  Without loss of generality, we work in the
upper half-plane model of $\bH^2$ and set $a_1=i$, $\ell = \{z\in
\bH^2 : |z| = R\}$, where we let $R\to +\infty$.  In this case
$d(a_1,\ell) = \log R$ and we want to keep $r=\varepsilon \sinh(\log
R) = \frac{\varepsilon(R^2-1)}{2R}$ constant, so we take
$\varepsilon = \frac{2rR}{R^2-1}$.  For $z\in \bH^2$,
\[
d(z,\ell)=\tfrac12 \,d(z,w),\text{ where }w=\frac{R^2}{\overline z}
=\text{reflection of }z\text{ across }\ell.
\]
Then the equation $\sinh d(z,a_1)=\varepsilon \sinh d(z, \ell)$
becomes
\[
\sinh\left(2\tanh^{-1}\left\vert\frac{z-i}{z+i}\right\vert\right)
= \frac{2rR}{R^2-1}\sinh\left(
\tanh^{-1}\left\vert\frac{z-\frac{R^2}{\overline z}}{z-\frac{R^2}{z}}
\right\vert\right) .
\]
The left-hand side simplifies to
\[
\frac{2\vert z+i\vert\,\vert z-i\vert}{\vert z+i\vert^2 -
 \vert z-i\vert^2} = \frac{\vert z^2+1\vert}{2\Im z}.
\]
On the right-hand side,
\[
\left\vert\frac{z-\frac{R^2}{\overline z}}{z-\frac{R^2}{z}}
\right\vert = \frac{R^2 - \vert z\vert^2}{\left\vert R^2 - \vert z\vert^2 
\frac{z}{\overline z}\right\vert} = \frac{R^2-a}{\sqrt{R^4
+a^2 -2R^2a\cos(\theta)}}
\]
where $a= \vert z\vert^2$ and $\theta = 2\arg z$.  Then
\[
\lim_{R\to \infty}
\frac{2rR}{R^2-1}\sinh\left(\tanh^{-1}\left(\frac{R^2-a}{\sqrt{R^4
+a^2 -2R^2\cos(\theta)}}\right)\right) 
=\frac{\sqrt{2}r}{|z|\sqrt{1-\cos(\theta)}}.
\]
Thus Definition \ref{def:focdirhypconic} gives for our circle
the equation
\[
\frac{\vert z^2+1\vert}{2\Im
  z}=\frac{\sqrt{2}r}{|z|\sqrt{1-\cos(\theta)}}
= \frac{r}{|z|}\csc(\theta/2)= \frac{r}{|z|}\frac{|z|}{\Im z} =
\frac{r}{\Im z},
\]
or
\begin{equation}
  \vert z^2+1\vert = {2r}.
  \label{eq:focdircircle}
\end {equation}   
in the upper half-plane.  This is an algebraic curve but not a metric
circle. Figure \ref{fig:circ} shows the case of $r=.25$ (in solid
color) as drawn with
\texttt{Mathematica}. This curve passes through the points
$i\sqrt{\frac{3}{2}}$, $i\sqrt{\frac{1}{2}}$, and $i \pm 
\sqrt{\frac{\sqrt{17}-4}{2}}$; the circle centered on the imaginary axis
tangent to it at $i\sqrt{\frac{3}{2}}$ and $i\sqrt{\frac{1}{2}}$
is shown with a dashed line in the same figure. The curves are close
but do not coincide.
\end{proof}
\begin{figure}[hbt]
  \begin{center}
    \includegraphics[width=2in]{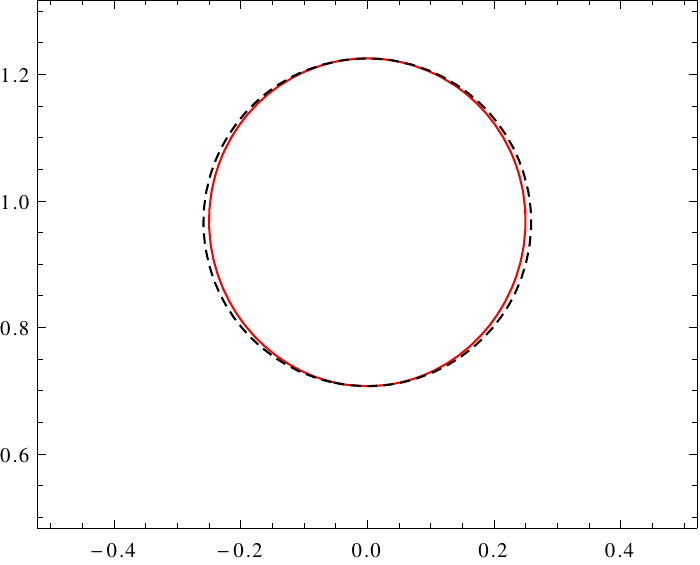}
    \end{center}
\caption{A ``circle''  with focus $i$ and $r=.25$ (solid color) and a
  tangent metric circle (dashed) in the upper half-plane}
\label{fig:circ}
\end{figure}  

Aside from circles, there are various other circle-like curves that
play a role in hyperbolic geometry.  These may be considered to be
conics according to certain definitions. Note also that they are
distinct from the circles of Definition \ref{def:focdirhypconic}.
\begin{definition}
\label{def:cycle}
A \emph{horocycle} (occasionally called a \emph{paracycle})
in the Poincar\'e disk model of $\bH^2$ is the
intersection of the disk with a circle tangent to the unit circle (and
lying inside the circle). A \emph{hypercycle} in the Poincar\'e disk
model of $\bH^2$ is the intersection of the disk with a circle meeting
the unit circle in exactly two points.  These have well-known
intrinsic definitions. A horocycle is the limit of a sequence
of circles $C_n$ (in the sense of Definition \ref{def:circ}) all
passing through a fixed point $x_0$, with
centers $x_n$ all lying on a fixed ray through $x_0$
and with radii $d(x_n,x_0)=r_n\to \infty$.  See Figure \ref{fig:horo}(a).
A hypercycle is a curve on one side of a given line $\ell$
whose points all have the same orthogonal distance from $\ell$.
See Figure \ref{fig:horo}(b).
Note that horocycles and hypercycles are clearly conics
in the sense of Definition \ref{def:alghypconic}. But they are not
covered by Definitions \ref{def:twofocushypconic} and
\ref{def:focdirhypconic}. Moln\'ar observes in \cite{MR487028} that
metric circles (Definition \ref{def:circ}), horocycles, and
hypercycles are all special cases of 
Definition \ref{def:Molnar} when the two foci coincide.
\end{definition}
  \begin{figure}[hbt]
  \begin{center}
    \includegraphics[width=2in]{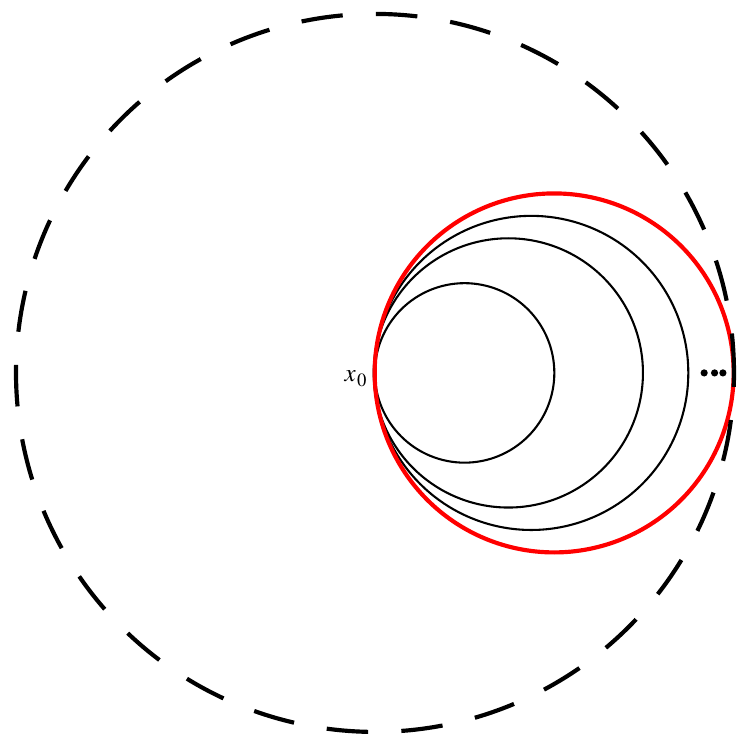}\hspace{1in}
    \includegraphics[width=2in]{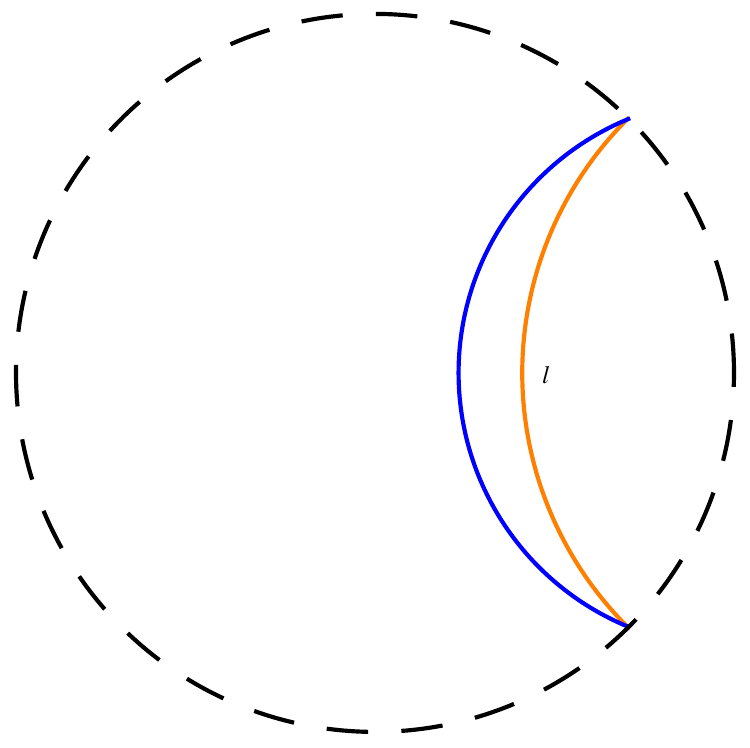}
    \end{center}
\caption{(a) (left) A horocycle (solid red) as a limit of circles (black)
  through $x_0$ with radii going to infinity. (b) (right) A hypercycle
  (solid blue) and a straight line $\ell$ (orange) 
  with the same ideal limits at infinity} 
\label{fig:horo}
\end{figure}

\subsection{Ellipses}
\label{sec:ellipse}

Next, we consider the case of the (noncircular) ellipse. There are two
main competing definitions: Definition \ref{def:twofocushypconic} and
Definition \ref{def:focdirhypconic}.
\begin{theorem}
\label{thm:twoellipses}
The definition of ellipse in Definition \ref{def:focdirhypconic} does
\emph{not always} agree with the definition of ellipse in Definition
\ref{def:twofocushypconic}.  However, there are cases where they
coincide. More precisely, when Definition \ref{def:focdirhypconic}
gives a closed curve in $\bH^2$, this curve is also a two-focus ellipse.
\end{theorem}
\begin{proof}
We will work in the upper half-plane model of $\bH^2$ and, without
loss of generality, put one focus at $i$ and let the imaginary axis be
an axis of the ellipse.  For an ellipse with the ``two-focus
definition'' and foci at $i$ and $bi$, $b>0$, the equation is
\[
2\tanh^{-1}\left({\frac{|z-i|}{|z+i|}}\right) +
2\tanh^{-1}\left({\frac{|z-bi|}{|z+bi|}}\right) = c,
\]
which can be rewritten as the algebraic equation
\begin{multline}
\label{eq:2focell}
\left(x^2+y^2+1+\sqrt{(x^2 - y^2 + 1)^2 + 4 x^2 y^2}\right)\\
\times \left(x^2+y^2+b^2+\sqrt{(x^2 - y^2 + b^2)^2 + 4 x^2 y^2}\right)
=4be^c\,y^2
\end{multline}
with $c>0$.  Plots of this equation for $b= \frac34$ and for
various values of $c$ are shown in Figure \ref{fig:2focell}.
The minimal value of $c$ to have the foci inside the ellipse is
the hyperbolic distance between the foci, or $\vert\log b\vert$.
As $c$ increases, the curves get bigger and bigger and look more like
circles. Now that since \eqref{eq:2focell} implies that $d(z,i)\le c$,
any ellipse in the sense of Definition \ref{def:twofocushypconic} is
automatically compact (closed) in $\bH^2$.
\begin{figure}
\begin{center}
\includegraphics[width=3in]{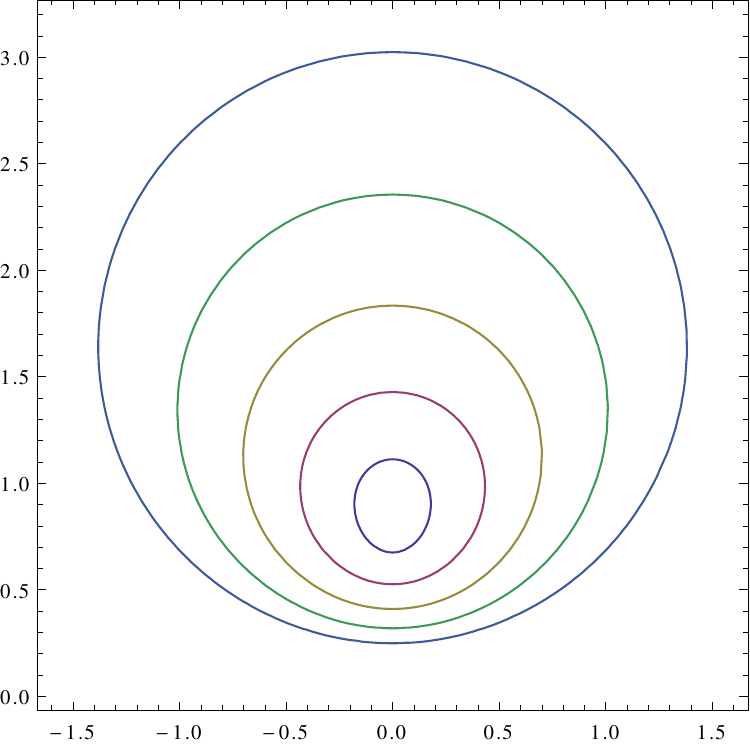}
\end{center}
\caption{Two-focus ellipses in the upper half-plane with foci at $i$
  and $\frac{3i}{4}$, as drawn with \texttt{Mathematica}} 
\label{fig:2focell}
\end{figure}

Now consider the focus-directrix definition for an ellipse in the
upper half-plane, with a focus at $i$ and directrix $|z|=r$, $r>1$ (this
choice makes the imaginary axis an axis of the ellipse).  The distance
from $z$ to the directrix is half the distance to the reflection of
$z$ across the directrix, which is $\frac{r^2}{\overline z}$.  Thus
the equation becomes
\[
\sinh\left(2\tanh^{-1}\left({\frac{|z-i|}{|z+i|}}\right)\right)
=\varepsilon\,\sinh \left(\tanh^{-1}
\left(\frac{\left\vert z-\frac{r^2}{\overline z}\right\vert}
  {\left\vert z-\frac{r^2}{z}\right\vert}\right)\right),
\]
which simplifies (after squaring both sides) to
\begin{multline}
\label{eq:focdirell}
r^2 \left(1 + x^4 + y^4 + 2 y^2 \left(-1 + \varepsilon^2\right) + 
    2 x^2 \left(1 + y^2 + \varepsilon^2\right)\right) \\=
    \varepsilon^2\,     \left(r^4 + \left(x^2 + 
      y^2\right)^2\right).
\end{multline}
This is a relatively simple quartic equation in $x$ and $y$, basically
the Cassini oval equation, and has
some interesting features.  For example, if one sets
$\varepsilon=\frac{1}{r}$, this reduces to a lemniscate passing
through $(0,0)$ (an ideal boundary point of $\bH^2$).  When
$\varepsilon > \frac{1}{r}$, the curve (viewed in $\bH^2$) is not
closed and approaches two distinct ideal boundary points.
Pictures of this behavior appear in Figure \ref{fig:focdirell}.
As a check that having two distinct ideal boundary points is not just
an artifact of the calculation,
one can check that upon substituting $r=3$ and $\varepsilon =
\frac12$ into \eqref{eq:focdirell}, one gets two points with $y=0$,
namely $x=\pm \sqrt{\frac37}$.

\begin{figure}
\begin{center}
  \includegraphics[width=2in]{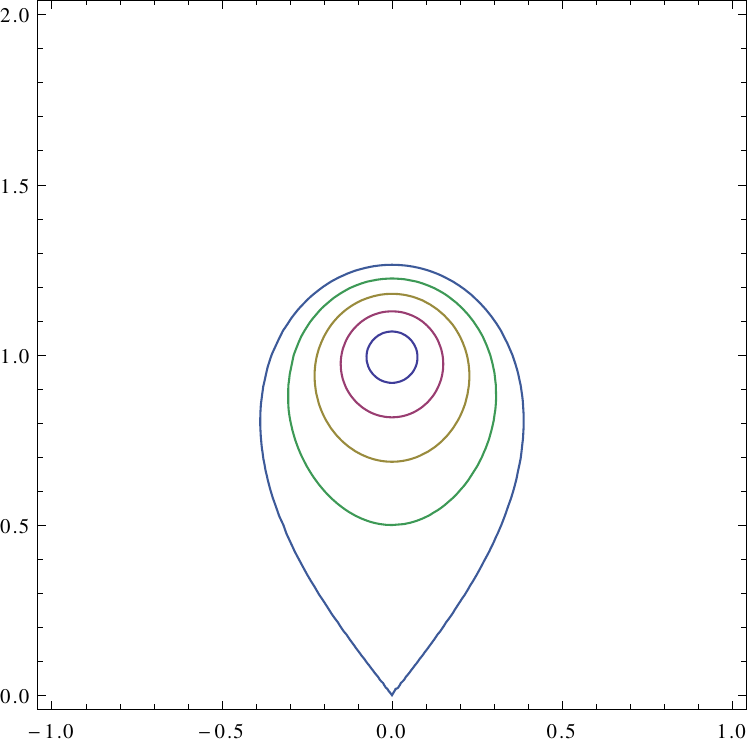}\qquad\qquad
  \includegraphics[width=2in]{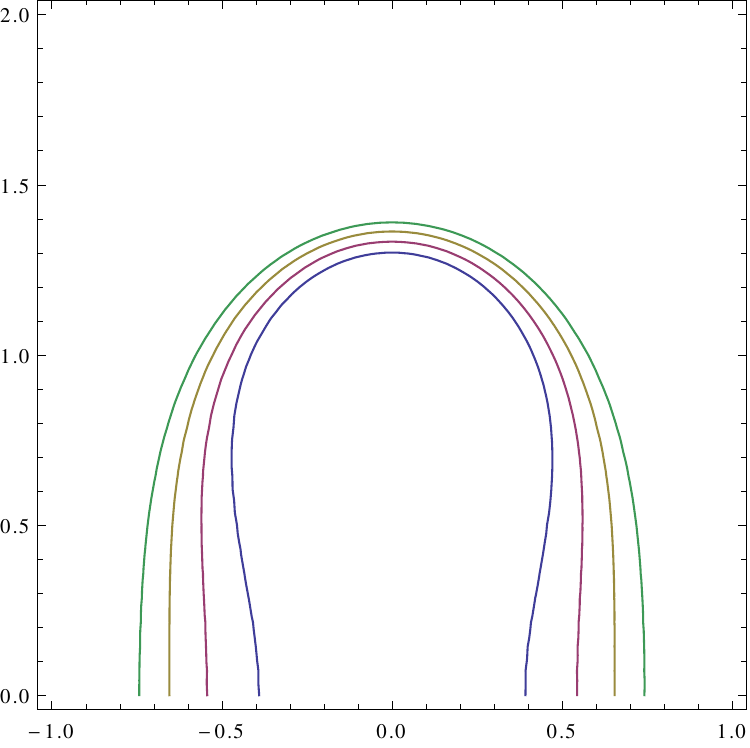}
\end{center}
\caption{Focus/directrix ellipses in the upper half-plane with focus at $i$
  and directrix $|z|=2$, as drawn with \texttt{Mathematica}.  On the
  left, cases with $\varepsilon \le 
  .5$. The case $\varepsilon =.5$ is a lemniscate.  On the right,
  cases with $\varepsilon$ from $.6$ to $.9$.}
\label{fig:focdirell}
\end{figure}

To illustrate another difference between the two definitions, consider the
case of the two-focus definition
when the foci coincide, i.e., $b=1$ in equation
\eqref{eq:2focell}.  Then equation \eqref{eq:2focell} reduces to
\[
x^2+y^2+1+\sqrt{(x^2 - y^2 + 1)^2 + 4 x^2 y^2}= 2e^{c/2}y
\]
or
\[
(x^2 - y^2 + 1)^2 + 4 x^2 y^2 - (2e^{c/2}y-x^2-y^2-1)^2 = 0,
\]
which simplifies to the equation of a circle:
\begin{equation}
x^2+\left(y-\cosh(c/2)\right)^2 = \sinh^2(c/2).
\label{eq:2foccirc}
\end{equation}
However, the focus/directrix equation \eqref{eq:focdirell} never
reduces to a circle.

However, perhaps rather surprisingly, focus/directrix ellipses with
$\varepsilon < \frac{1}{r}$ (this is the case where the curve is closed)
turn out to be special cases of two-focus ellipses.
A rather horrendous calculation
with \texttt{Mathematica} or \texttt{MuPAD}
shows for example that \eqref{eq:2focell}
with $b=2$ and $c=\log\left(\frac52\right)$ is equivalent to
\eqref{eq:focdirell} with
\[
\varepsilon=\frac{\sqrt{209}}{21}, \quad r=\sqrt{\frac{11}{19}}.
\]
To see this, rewrite \eqref{eq:focdirell}  in the form
\[
\begin{aligned}
x^2+y^2+1&+\sqrt{(x^2 - y^2 + 1)^2 + 4 x^2 y^2}\\ &=
\frac{20y^2}{x^2+y^2+4+\sqrt{(x^2 - y^2 + 4)^2 + 4 x^2 y^2}}\\
&=\frac{20y^2\left(x^2+y^2+4-\sqrt{(x^2 - y^2 + 4)^2 + 4 x^2 y^2}\right)}
{(x^2+y^2+4)^2-((x^2 - y^2 + 4)^2 + 4 x^2 y^2)},
\end{aligned}
\]
simplify, and rewrite in the form $E + \sqrt{B} = F\sqrt{D}$,
where $B=(x^2 - y^2 + 1)^2 + 4 x^2 y^2$ and $D=(x^2 - y^2 + 4)^2 + 4 x^2 y^2$.
Square both sides, again simplify and regroup to get the
term with $\sqrt{B}$ by itself, and finally square again.
After factoring out $y^2$, one finally ends up with the equation
\[
20x^4 + 40x^2y^2 + 325x^2 + 20y^4 - 116y^2 + 80 = 0,
\]
which is equivalent to \eqref{eq:focdirell} for the given parameters.
Other values of $r$ and $\varepsilon$
(with $r\varepsilon < 1$) can be handled similarly; one just
needs to solve for the values of $b$ and $c$ giving the same $y$-intercepts.
\end{proof}

\subsection{Parabolas}
\label{sec:parabola}
Next, we consider the case of the parabola. Here the result is rather simple:
\begin{theorem}
\label{thm:parabolas}
The definitions of parabolas in Definition \ref{def:focdirhypconic}
and  in Definition \ref{def:twofocushypconic}
never agree.  In all cases, however, a parabola in $\bH^2$ is not
closed.
\end{theorem}
\begin{proof}
Without loss of generality, we can again use the Poincar\'e upper
half-plane model of $\bH^2$ and put one focus at $i$ and take
the axis of the parabola to be the imaginary axis.  The two-focus 
definition of  Definition \ref{def:twofocushypconic} is the limiting 
case of \eqref{eq:2focell} as we keep $be^c = C/2$ fixed and 
let $b\to 0_+$.  (This is because $d(i,ib)=\vert\log b\vert = - \log b$ 
for $0<b<1$ and we want $c-d(i,ib)= c + \log b$ to be held constant.)  
Then \eqref{eq:2focell} reduces to 
\begin{equation}
\label{eq:2fopar}
\left(x^2+y^2+1+\sqrt{(x^2 - y^2 + 1)^2 + 4 x^2 y^2}\right)
\left(x^2+y^2\right) = C\,y^2,
\end{equation}
or equivalently (after regrouping and squaring to get rid of the radical, 
then factoring out a $y^2$)
\begin{equation}
\label{eq:2fopara}
2(C - 2)(x^2 + y^2)^2  + 2C(x^2 + y^2) - C^2 y^2 =0 .
\end{equation}
This is the equation of a lemniscate through the origin.  (Remember 
that $0$, however, is only an ideal boundary point of $\bH^2$.)  
Definition \ref{def:focdirhypconic} simply gives \eqref{eq:focdirell} 
with $\varepsilon = 1$, which reduces to
\begin{equation}
\label{eq:fodirpar}
1 - r^2 + 4 x^2 + \left(1 - \tfrac{1}{r^2}\right) (x^2 + y^2)^2 = 0  ,
\end{equation}
which is a Cassini oval equation.  Note that \eqref{eq:2fopara} and 
\eqref{eq:fodirpar} never agree, since for $r\ne 1$ (we don't want 
the directrix of the parabola to pass through the focus), the curve 
given by \eqref{eq:fodirpar} doesn't pass through the origin. Pictures 
of the various kinds of parabolas, plotted by \texttt{Mathematica}, 
are shown in Figures \ref{fig:parabolas1} and \ref{fig:parabolas2}.
\end{proof}
\begin{figure}[hbt]
\begin{center}
  \includegraphics[width=2in]{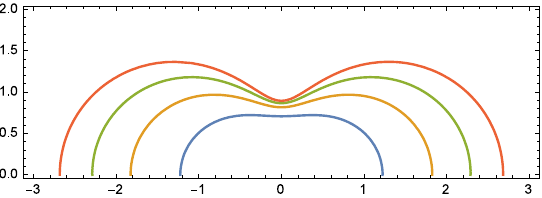}\qquad\qquad
  \includegraphics[width=2in]{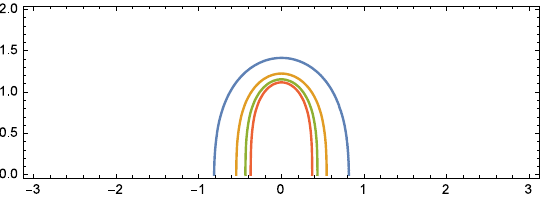}
\end{center}
\caption{Focus/directrix parabolas in the upper half-plane with focus at $i$
and directrix $|z|=r$, as drawn with \texttt{Mathematica}.  On the
left, cases with $r < 1$. These are Cassini ovals.  On the right,
cases with $r>1$. Of course, if one were wearing ``hyperbolic glasses,''
all would look roughly the same.}
\label{fig:parabolas1}
\end{figure}
\begin{figure}[hbt]
\begin{center}
\includegraphics[width=2in]{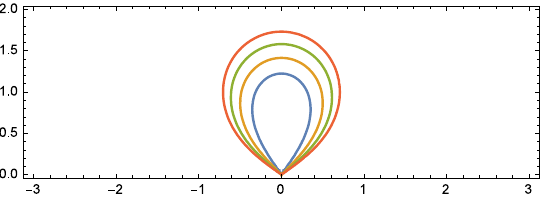}
\end{center}
\caption{Two-focus parabolas in the upper half-plane with focus at $i$,
  as drawn with \texttt{Mathematica}.  Note the lemniscate shape.}
\label{fig:parabolas2}
\end{figure}

\subsection{Hyperbolas}
\label{sec:hyperbola}
Finally, we consider the case of the hyperbola.
\begin{theorem}
\label{thm:twohyps}
The definition of hyperbola in Definition \ref{def:focdirhypconic} does
\emph{not always} agree with the definition of hyperbola in Definition
\ref{def:twofocushypconic}.  However, the two-focus hyperbola
from Definition \ref{def:twofocushypconic} is a special case of
the focus-directrix hyperbola of Definition \ref{def:focdirhypconic}.
\end{theorem}
\begin{proof}
Consider the two-focus hyperbola. Fix $c>0$.
(When $c=0$, the definition degenerates to the bisector of the line
segment joining the two foci, which is a straight line (i.e., a geodesic).)
We will work in the upper half-plane model of $\bH^2$ and, without
loss of generality, put one focus at $i$ and the other focus at $ib$, $b>1$.
The equation of the two-focus hyperbola is then
\[
2\tanh^{-1}\left({\frac{|z-i|}{|z+i|}}\right) -
2\tanh^{-1}\left({\frac{|z-bi|}{|z+bi|}}\right) = \pm c,
\]
which can be rewritten as the algebraic equation
\begin{multline}
  \label{eq:2fochyp}
b^2 + x^2 + y^2 + \sqrt{b^4 + 2 b^2 (x^2 - y^2) + (x^2 + y^2)^2} =\\
b e^{\pm c} \left(1 + x^2 + y^2 + \sqrt{1 + 2 (x^2 - y^2)
  + (x^2 + y^2)^2}\right)  
\end{multline}
with $c>0$.  Note that the hyperbola should intersect its axis
(here the imaginary axis) at two points of the form $iy$, $1< y < b$,
so we want $0 < c < \log b$, and the two $y$-intercepts are at
$i\sqrt{be^{\pm c}}$.  Comparing this with the $y$-intercepts for the
focus-directrix hyperbola \eqref{eq:focdirell} (the equation is the
same as for the ellipse --- the only difference is the value of the
eccentricity $\varepsilon$), we see that this agrees with a
focus-directrix hyperbola with parameters satisfying
\[
\frac{\sqrt{r + r^2 \varepsilon}}{\sqrt{r + \varepsilon}} =
\frac{\sqrt{b}}{e^{c/2}}, \quad \frac{\sqrt{r - r^2\varepsilon}}
     {\sqrt{r - \varepsilon}} = \sqrt{b}\,e^{c/2},
     \]
or
\begin{equation}
r =\sqrt{\frac{-b + 2 b^2 e^c - b e^{2 c}}
  {b - 2 e^c + b e^{2 c}}},\quad  \varepsilon =
\frac{\sqrt{b (-1 + 2 b e^c - e^{2 c})(b - 2 e^c + b e^{2 c})}}
     {b(e^{2c}-1)}.
\label{eq:revals}
\end{equation}
Note that since $c < \log b$, the value of $\varepsilon$ is $>1$.
Just as an example, if $b=2$ and $c = \log(3/2)$, after removing some
superfluous factors, equation \eqref{eq:2fochyp} reduces to
$24 + 6 x^4 - 26 y^2 + 6 y^4 + 3 x^2 (-17 + 4 y^2) = 0$, which
agrees with the focus-directrix hyperbola with focus $i$, directrix
$|z|=\sqrt{\frac{11}{7}}$, and eccentricity $\varepsilon=\frac{\sqrt{77}}{5}$.
A graph of this hyperbola, drawn with \texttt{Mathematica}, appears in Figure
\ref{fig:hyp}.

\begin{figure}[hbt]
\begin{center}
\includegraphics[width=3in]{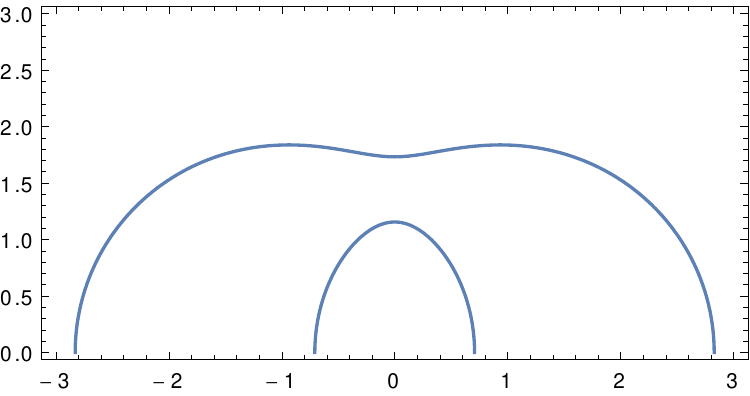}
\end{center}
\caption{A hyperbola in the upper half-plane with foci at $i$ and $2i$,
  $c=\log(3/2)$, as drawn with \texttt{Mathematica}.}
\label{fig:hyp}
\end{figure}

So this analysis shows that every two-focus hyperbola is also a
focus-directrix hyperbola.  The converse fails, however. Indeed,
one can see from \eqref{eq:focdirell} that the focus-directrix hyperbola
with $r=\varepsilon>1$ degenerates to the equation
\[
(r^2 + 1) x^2 + (r^2 - 1) y^2  = \frac{r^4 - 1}{2},
\]
which, surprisingly, is an \emph{ellipse} in Cartesian coordinates.
This has only one $y$-intercept in the upper half-plane, at
the point $i\sqrt{\frac{r^2 + 1}{2}}$.  So this ``hyperbola'' has only
one vertex, the other vertex having gone to $+\infty i$, and this
cannot be written as a two-focus hyperbola.
\end{proof}

\bibliographystyle{amsplain}
\bibliography{conics}

\providecommand{\bysame}{\leavevmode\hbox to3em{\hrulefill}\thinspace}
\providecommand{\MR}{\relax\ifhmode\unskip\space\fi MR }
\providecommand{\MRhref}[2]{%
  \href{http://www.ams.org/mathscinet-getitem?mr=#1}{#2}
}
\providecommand{\href}[2]{#2}
\begin{thebibliography}{1}

\bibitem{MR510197}
Lars~V. Ahlfors, \emph{Complex analysis}, third ed., McGraw-Hill Book Co., New
  York, 1978, An introduction to the theory of analytic functions of one
  complex variable, International Series in Pure and Applied Mathematics.
  \MR{510197 (80c:30001)}

\bibitem{MR0346643}
Friedrich Bachmann, \emph{Aufbau der {G}eometrie aus dem {S}piegelungsbegriff},
  Springer-Verlag, Berlin-New York, 1973, Zweite erg{\"a}nzte Auflage, Die
  Grundlehren der mathematischen Wissenschaften, Band 96. \MR{0346643 (49
  \#11368)}

\bibitem{MR1628013}
H.~S.~M. Coxeter, \emph{Non-{E}uclidean geometry}, sixth ed., MAA Spectrum,
  Mathematical Association of America, Washington, DC, 1998, first edition
  published 1942. \MR{1628013 (99c:51002)}

\bibitem{MR3274526}
G{\'e}za Csima and Jen{\H{o}} Szirmai, \emph{Isoptic curves of conic sections
  in constant curvature geometries}, Math. Commun. \textbf{19} (2014), no.~2,
  277--290. \MR{3274526}

\bibitem{2015arXiv150406450}
\bysame, \emph{Isoptic curves of generalized conic sections in the hyperbolic
  plane}, arXiv:1504.06450, 2015.

\bibitem{MR0095442}
Kuno Fladt, \emph{Die allgemeine {K}egelschnittsgleichung in der ebenen
  hyperbolischen {G}eometrie. {II}}, J. Reine Angew. Math. \textbf{199} (1958),
  203--207. \MR{0095442 (20 \#1944)}

\bibitem{MR0171205}
\bysame, \emph{Elementare {B}estimmung der {K}egelschnitte in der
  hyperbolischen {G}eometrie}, Acta Math. Acad. Sci. Hungar. \textbf{15}
  (1964), 247--257. \MR{0171205 (30 \#1436)}

\bibitem{MR487028}
E.~Moln{\'a}r, \emph{Kegelschnitte auf der metrischen {E}bene}, Acta Math.
  Acad. Sci. Hungar. \textbf{31} (1978), no.~3-4, 317--343. \MR{487028
  (81e:51006)}

\bibitem{MR1505334}
William~E. Story, \emph{On {N}on-{E}uclidean {P}roperties of {C}onics}, Amer.
  J. Math. \textbf{5} (1882), no.~1-4, 358--381. \MR{1505334}

\end{thebibliography}
\end{document}